\documentclass[12pt]{article}

\usepackage{amsmath, epsfig, cite}
\usepackage{amssymb}
\usepackage{amsfonts}
\usepackage{latexsym}
\usepackage{float}
\usepackage{color}
\usepackage{booktabs}
\usepackage{graphicx}
\usepackage{caption}
\usepackage[colorlinks=true,linkcolor=blue,citecolor=blue,urlcolor=blue]{hyperref}

\newtheorem{thm}{Theorem}

\newtheorem{lem}{Lemma}[section]

\newcommand{\pf}{\noindent{\it Proof.} }

\newcommand{\NN}{\mathbb{Z}_{\geq 0}}
\newcommand{\NNq}{\mathbb{Z}_{\geq 0}[q]}
\newcommand{\qbinom}[2]{\genfrac{[}{]}{0pt}{}{#1}{#2}}
\newcommand{\qpoch}[1]{(q;q)_{#1}}
\newcommand{\bc}[1]{\binom{#1}{2}}

\numberwithin{equation}{section}

\newcommand{\qed}{\hfill$\square$\medskip}

\begin{document}

\begin{center}
{\Large\bf Proofs of Two Positivity Conjectures of Guo}
\end{center}

\vskip 2mm \centerline{Ji-Cai Liu}
\begin{center}
{\footnotesize Department of Mathematics, Wenzhou University, Wenzhou 325035, PR China\\[5pt]
{\tt jcliu2016@gmail.com} \\[10pt]
}
\end{center}

\vskip 0.7cm \noindent{\bf Abstract.}
We prove two positivity conjectures proposed by Guo for alternating sums and factorial ratios built from Gaussian coefficients.  The first result proves the positivity of the odd $q$-super Catalan numbers
\[
   C_{m,n}(q)=\frac{[2m+1]![2n]!}{[m+n+1]![m]![n]!}.
\]
The proof uses the positivity theorem of Warnaar and Zudilin for the usual $q$-super Catalan numbers, together with two recurrences obtained from a double application of the $q$-Chu--Vandermonde summation.  The second result proves Guo's conjectural strengthening of his alternating-sum positivity theorem, replacing the exponent coefficient $2r-1$ by every odd coefficient $2b-1$, $1\leq b\leq r$.  Its proof combines a $q\mapsto q^{-1}$ reciprocity with a finite deletion recurrence.


\vskip 2mm
\noindent{\it MR Subject Classifications}: 05A10, 05A30, 33D15

\section{Introduction}

For a non-negative integer $n$ define
\[
   [n]=\frac{1-q^n}{1-q},\qquad [n]!=[1][2]\cdots[n],\qquad [0]!=1.
\]
For integers $N,K$ we use the Gaussian coefficient
\begin{equation}\label{eq:qbinom-def}
   \qbinom{N}{K}=\begin{cases}
   \dfrac{[N]!}{[K]![N-K]!},&0\leq K\leq N,\\[6pt]
   0,&\hbox{otherwise.}
   \end{cases}
\end{equation}
We write $P(q)\in\NNq$ if $P(q)$ is a polynomial in $q$ with non-negative integer coefficients.  We also put
\[
   \qpoch{n}=(1-q)(1-q^2)\cdots(1-q^n),\qquad \qpoch{0}=1,
\]
and use the standard convention that $1/\qpoch{n}=0$ if $n<0$.

The positivity of Gaussian coefficients is a basic model for positivity questions involving ratios of $q$-factorials.  Warnaar and Zudilin~\cite{WarnaarZudilin} formulated a broad conjecture for positive $q$-factorial ratios
\[
   D(\mathbf a,\mathbf b;q)=\frac{[a_1]!\cdots [a_r]!}{[b_1]!\cdots [b_s]!}
\]
under the usual Landau-type step-function condition.  They proved, among other examples, the positivity of the $q$-super Catalan numbers
\begin{equation}\label{eq:A-def}
   A_{m,n}(q)=\frac{[2m]![2n]!}{[m+n]![m]![n]!},\qquad m,n\geq 0,
\end{equation}
and of the related ratios
\[
   B_{n,m}(q)=\frac{[2n]![m]!}{[n]![2m]![n-m]!},\qquad n\geq m\geq 0.
\]
They also emphasized that positivity of alternating sums of products of Gaussian coefficients is usually delicate.

Guo~\cite{Guo} studied factors of alternating sums of products of $q$-binomial coefficients.  Motivated by Warnaar and Zudilin's theorem on \eqref{eq:A-def} and by integrality results of Xia~\cite{Xia}, Guo introduced the odd analogue
\begin{equation}\label{eq:C-def}
   C_{m,n}(q)=\frac{[2m+1]![2n]!}{[m+n+1]![m]![n]!}
\end{equation}
and conjectured that $C_{m,n}(q)$ is positive.  Guo noted that this positivity is not, in general, a direct consequence of the Warnaar--Zudilin factorial-ratio conjecture, because the tuples $(2m+1,2n)$ and $(m,n,m+n+1)$ need not satisfy the relevant Landau condition.

Guo also proved a positivity theorem for a family of alternating sums.  Let $m_1,\ldots,m_r$ and $n_1,\ldots,n_s$ be non-negative integers, with cyclic conventions
\[
   m_{r+1}=m_1,\qquad n_{s+1}=n_1.
\]
For $1\leq b\leq r$ and $0\leq a\leq s$, define
\begin{align}\label{eq:F-def}
F_{r,s}^{(b)}(\mathbf m;\mathbf n;a;q)
&=\frac{[m_1]![n_1]![m_r+n_s+1]!}{[m_1+m_r+1]![n_1+n_s]!}
\sum_{k=-n_1}^{n_1}(-1)^k q^{a k^2+(2b-1)\bc{k}}        \\
&\quad\times
\prod_{i=1}^{r}\qbinom{m_i+m_{i+1}+1}{m_i+k}
\prod_{j=1}^{s}\qbinom{n_j+n_{j+1}}{n_j+k} .\nonumber
\end{align}
Guo's theorem \cite[Theorem 1.3]{Guo} is the case $b=r$ of the assertion
\begin{equation}\label{eq:F-positive}
   F_{r,s}^{(b)}(\mathbf m;\mathbf n;a;q)\in\NNq.
\end{equation}
He conjectured that \eqref{eq:F-positive} remains true for all $1\leq b\leq r$.

The purpose of this paper is to prove these two conjectures \cite[Conjectures 5.1 and 7.1]{Guo}.  Our two main results are the following.

\begin{thm}\label{thm:odd-Catalan}
For all $m,n\in\NN$, the odd $q$-super Catalan number $C_{m,n}(q)$ belongs to $\NNq$.
\end{thm}

\begin{thm}\label{thm:Guo71}
For $r,s>1$, let $m_1,\ldots,m_r$ and $n_1,\ldots,n_s$ be positive integers, with $m_{r+1}=m_1$ and $n_{s+1}=n_1$.  For all integers $a,b$ with $0\leq a\leq s$ and $1\leq b\leq r$, we have
\[
   F_{r,s}^{(b)}(\mathbf m;\mathbf n;a;q)\in\NNq.
\]
\end{thm}

Section~\ref{sec:proof-thm1} proves Theorem~\ref{thm:odd-Catalan}.  Section~\ref{sec:proof-thm2} proves Theorem~\ref{thm:Guo71}.  The last section gives some concluding remarks.

\section{Proof of Theorem~\ref{thm:odd-Catalan}}\label{sec:proof-thm1}
Guo \cite{Guo} established the following result through a double application of the $q$-Chu--Vandermonde summation (see \cite[Appendix (II.7)]{GasperRahman}):

\begin{lem}[Guo]\label{lem:double-expansion}
For $N\geq 0$ and $h\geq 1$, we have
\begin{align}\label{eq:double-expansion}
   \qbinom{2N+2h}{h-1}
   ={}&\sum_{k=0}^{\lfloor (h-1)/2\rfloor}\sum_{j=k}^{h-k-1}
      q^{k(N+k+1)+j(N+j+1)}    \\
   &\quad\times
      \qbinom{N+h}{j}\qbinom{j}{k}\qbinom{N+h-j}{h-j-k-1} .\nonumber
\end{align}
\end{lem}

Multiplying \eqref{eq:double-expansion} by $[2N]![h]!/([N]![N+h]!)$, Guo \cite{Guo} obtained the following recurrence for $C_{N+h,N}(q)$.
\begin{lem}[Guo]\label{lem:forward-recurrence}
For $N\geq0$ and $h\geq1$, we have
\begin{align}\label{eq:forward-recurrence}
   C_{N+h,N}(q)
   ={}& q^h A_{N+h,N}(q)  \\
   &+\sum_{k=0}^{\lfloor (h-1)/2\rfloor} C_{k,N}(q)
      \sum_{j=k}^{h-k-1}
      q^{k(N+k+1)+j(N+j+1)}
      \qbinom{h}{2k+1}\qbinom{h-2k-1}{j-k}.\nonumber
\end{align}
\end{lem}

We also need the other recurrence for $C_{N,N+h}(q)$.
\begin{lem}\label{lem:backward-recurrence}
For $N\geq0$ and $h\geq1$, we have
\begin{align}\label{eq:backward-recurrence}
   C_{N,N+h}(q)
   ={}&\sum_{k=0}^{\lfloor (h-1)/2\rfloor} C_{N,k}(q)
      \sum_{j=k}^{h-k-1}
      q^{k(N+k+1)+j(N+j+1)}   \\
   &\quad\times
      \qbinom{h-1}{2k}\qbinom{h-2k-1}{j-k}.\nonumber
\end{align}
\end{lem}

\pf
Multiply \eqref{eq:double-expansion} by $[2N+1]![h-1]!/([N]![N+h]!)$.  The left-hand side is
\[
   \frac{[2N+1]![2N+2h]!}{[N]![N+h]![2N+h+1]!}=C_{N,N+h}(q).
\]
For a right-hand summand, direct cancellation gives
\begin{align*}
&\frac{[2N+1]![h-1]!}{[N]![N+h]!}
 \qbinom{N+h}{j}\qbinom{j}{k}\qbinom{N+h-j}{h-j-k-1} \\
&\qquad =
\frac{[2N+1]![2k]!}{[N+k+1]![N]![k]!}
\frac{[h-1]!}{[2k]![j-k]![h-j-k-1]!} \\
&\qquad = C_{N,k}(q)\qbinom{h-1}{2k}\qbinom{h-2k-1}{j-k}.
\end{align*}
Substitution into \eqref{eq:double-expansion} proves \eqref{eq:backward-recurrence}.\qed

\noindent{\it Proof of Theorem~\ref{thm:odd-Catalan}.}
We prove the assertion by strong induction on $m+n$.  The initial case is $C_{0,0}(q)=1$.  Assume $m+n>0$ and that the theorem has been proved for all pairs with smaller sum.

If $m=n$, then
\[
   C_{n,n}(q)=\frac{[2n+1]![2n]!}{[2n+1]![n]![n]!}
             =\qbinom{2n}{n}\in\NNq.
\]

Suppose next that $m>n$, and write $m=n+h$ with $h\geq1$.  By Lemma~\ref{lem:forward-recurrence}, the polynomial $C_{m,n}(q)$ is the sum of $q^hA_{m,n}(q)$ and terms of the form
\[
   q^{k(n+k+1)+j(n+j+1)} C_{k,n}(q)
      \qbinom{h}{2k+1}\qbinom{h-2k-1}{j-k}.
\]
Warnaar and Zudilin's positivity theorem gives $A_{m,n}(q)\in\NNq$.  Moreover $k+n<m+n$ in every summand, so the induction
hypothesis gives $C_{k,n}(q)\in\NNq$.  The remaining factors are powers of $q$
and Gaussian coefficients, and hence lie in $\NNq$.

Finally suppose that $n>m$, and write $n=m+h$ with $h\geq1$.  Lemma~\ref{lem:backward-recurrence} expresses $C_{m,n}(q)$ as a finite sum of terms of the form
\[
   q^{k(m+k+1)+j(m+j+1)} C_{m,k}(q)
      \qbinom{h-1}{2k}\qbinom{h-2k-1}{j-k}.
\]
Here $m+k<m+n$, so the induction hypothesis applies to $C_{m,k}(q)$, and all remaining factors are in $\NNq$.  This completes the induction.\qed

\section{Proof of Theorem~\ref{thm:Guo71}}\label{sec:proof-thm2}

The proof uses two elementary lemmas.  The first is a reciprocity under $q\mapsto q^{-1}$; the second is a deletion recurrence for two adjacent $m$-parameters.  We shall use Guo's theorem that
\begin{equation}\label{eq:Guo-theorem-used}
   F_{r,s}^{(r)}(\mathbf m;\mathbf n;a;q)\in\NNq
   \qquad (0\leq a\leq s)
\end{equation}
for the non-negative parameters occurring in the recurrence below.

We shall also use the elementary inversion identities
\begin{equation}\label{eq:q-inversion}
   [N]!_{q^{-1}}=q^{-\bc{N}}[N]!_q,
   \qquad
   \qbinom{N}{K}_{q^{-1}}=q^{-K(N-K)}\qbinom{N}{K}_q .
\end{equation}

For cyclic vectors $\mathbf m=(m_1,\ldots,m_r)$ and $\mathbf n=(n_1,\ldots,n_s)$ set
\begin{align}\label{eq:Delta-def}
\Delta(\mathbf m;\mathbf n)
&=\bc{m_1}+\bc{n_1}+\bc{m_r+n_s+1}
  -\bc{m_1+m_r+1}-\bc{n_1+n_s}  \\
&\quad +\sum_{i=1}^{r}m_i(m_{i+1}+1)+\sum_{j=1}^{s}n_jn_{j+1}.
\nonumber
\end{align}

\begin{lem}[Reciprocity]\label{lem:reciprocity}
For $0\leq a\leq s$ and $1\leq b\leq r$,
\begin{equation}\label{eq:reciprocity}
F_{r,s}^{(b)}(\mathbf m;\mathbf n;a;q^{-1})
=q^{-\Delta(\mathbf m;\mathbf n)}
F_{r,s}^{(r-b+1)}(\mathbf m;\mathbf n;s-a;q).
\end{equation}
Moreover, if $F_{r,s}^{(r-b+1)}(\mathbf m;\mathbf n;s-a;q)$ is a polynomial, then its degree is at most $\Delta(\mathbf m;\mathbf n)$.
\end{lem}

\pf
Using \eqref{eq:q-inversion}, under $q\mapsto q^{-1}$ the prefactor in \eqref{eq:F-def} contributes
\[
q^{-\bc{m_1}-\bc{n_1}-\bc{m_r+n_s+1}+
\bc{m_1+m_r+1}+\bc{n_1+n_s}}.
\]
The $i$th $m$-factor contributes
\[
   q^{-(m_i+k)(m_{i+1}+1-k)},
\]
and the $j$th $n$-factor contributes
\[
   q^{-(n_j+k)(n_{j+1}-k)}.
\]
Using cyclicity, we have
\[
\sum_{i=1}^{r}(m_i+k)(m_{i+1}+1-k)
=\sum_{i=1}^{r}m_i(m_{i+1}+1)+rk-rk^2
\]
and
\[
\sum_{j=1}^{s}(n_j+k)(n_{j+1}-k)
=\sum_{j=1}^{s}n_jn_{j+1}-sk^2.
\]
Consequently the total exponent after inversion equals
\[
   -\Delta(\mathbf m;\mathbf n)+(s-a)k^2+(2r-2b+1)\bc{k},
\]
which is precisely the exponent in
$F_{r,s}^{(r-b+1)}(\mathbf m;\mathbf n;s-a;q)$ after extracting the factor $q^{-\Delta(\mathbf m;\mathbf n)}$.  This proves \eqref{eq:reciprocity}.

For the degree bound, we use the degree of a rational function in $q$.
For a non-zero rational function $R(q)$, write $\deg R$ for the difference
between the degrees of its numerator and denominator after cancellation.  The
prefactor in the $k$th summand of
$F_{r,s}^{(r-b+1)}(\mathbf m;\mathbf n;s-a;q)$ contributes the first line of
\eqref{eq:Delta-def}; the $m$- and $n$-products contribute
\[
   \sum_{i=1}^{r}(m_i+k)(m_{i+1}+1-k)
   +\sum_{j=1}^{s}(n_j+k)(n_{j+1}-k).
\]
After adding the exponent $(s-a)k^2+(2r-2b+1)\bc{k}$, the rational-function
degree of this summand is
\[
   \Delta(\mathbf m;\mathbf n)-a k^2-(2b-1)\bc{k}\leq \Delta(\mathbf m;\mathbf n),
\]
because $a\geq0$, $b\geq1$, and $\bc{k}\geq0$ for every integer $k$.
Hence every non-zero summand has rational-function degree at most
$\Delta(\mathbf m;\mathbf n)$.  If the whole sum is a polynomial, its ordinary
polynomial degree cannot exceed this common upper bound.\qed

By using the $q$-Chu--Vandermonde summation (see \cite[Appendix (II.7)]{GasperRahman}), Guo \cite{Guo} established the following product identity.
\begin{lem}[Guo]\label{lem:product-identity}
For $m_1,m_2\in\NN$ and every integer $k$,
\begin{align}\label{eq:product-qchu}
&\qbinom{m_1+m_2+1}{m_1+k}
\qbinom{m_1+m_2+1}{m_2+k}       \\
&\quad=
\sum_{t=0}^{m_1-k+1}q^{t^2+2kt-t}
\frac{\qpoch{m_1+m_2+1}}
{\qpoch{t}\qpoch{t+2k-1}\qpoch{m_1-k-t+1}\qpoch{m_2-k-t+1}}.
\nonumber
\end{align}
Here the convention $1/\qpoch{u}=0$ for $u<0$ is used.
\end{lem}

We also need a deletion recurrence for $F_{r,s}^{(b)}(\mathbf m;\mathbf n;a;q)$.
\begin{lem}[Deletion recurrence]\label{lem:deletion}
Let $r\geq3$ and $2\leq b\leq r$.  Then
\begin{align}\label{eq:deletion}
F_{r,s}^{(b)}(m_1,m_2,\ldots,m_r;\mathbf n;a;q)
&=\sum_{\ell=0}^{m_1}q^{\ell^2+\ell}
\qbinom{m_1}{\ell}\qbinom{m_2+m_3+1}{m_2-\ell} \\
&\quad\times
F_{r-1,s}^{(b-1)}(\ell,m_3,\ldots,m_r;\mathbf n;a;q).
\nonumber
\end{align}
\end{lem}

\pf
Let
\[
C(\mathbf m;\mathbf n;k)=
\prod_{i=1}^{r}\qbinom{m_i+m_{i+1}+1}{m_i+k}
\prod_{j=1}^{s}\qbinom{n_j+n_{j+1}}{n_j+k}.
\]
The two factors involving $m_1$ and $m_2$ may be separated as
\begin{align}\label{eq:C-separation}
C(m_1,\ldots,m_r;\mathbf n;k)
&=\frac{\qpoch{m_2+m_3+1}\qpoch{m_r+m_1+1}}
{\qpoch{m_1+m_2+1}\qpoch{m_r+m_3+1}}  \\
&\quad\times
\qbinom{m_1+m_2+1}{m_1+k}
\qbinom{m_1+m_2+1}{m_2+k}
C(m_3,\ldots,m_r;\mathbf n;k),\nonumber
\end{align}
where $C(m_3,\ldots,m_r;\mathbf n;k)$ denotes the corresponding product with the cyclic $m$-vector $(m_3,\ldots,m_r)$.

Apply Lemma~\ref{lem:product-identity} to the two Gaussian coefficients in
\eqref{eq:C-separation}.  Substitute \eqref{eq:C-separation} and
\eqref{eq:product-qchu} into \eqref{eq:F-def}, and set
\[
   \ell=t+k-1.
\]
By the convention $1/\qpoch{u}=0$ for $u<0$, all terms with inadmissible indices vanish.  Since $\ell=t+k-1$, the remaining outer index satisfies $0\leq\ell\leq m_1$ after the order of summation is changed.  The exponent separates because
\begin{equation}\label{eq:exponent-separation}
   t^2+2kt-t+a k^2+(2b-1)\bc{k}
   =\ell^2+\ell+a k^2+(2b-3)\bc{k}.
\end{equation}

The remaining $k$-sum is converted into the defining sum for the smaller
alternating sum by the following comparison of the cyclic products:
\begin{align}\label{eq:C-recombine}
C(m_3,\ldots,m_r;\mathbf n;k)
&=\frac{\qpoch{\ell-k+1}\qpoch{\ell+k}\qpoch{m_r+m_3+1}}
{\qpoch{m_3+\ell+1}\qpoch{m_r+\ell+1}} \\
&\quad\times
C(\ell,m_3,\ldots,m_r;\mathbf n;k).\nonumber
\end{align}
Indeed, the two cyclic products have the same $n$-factors and the same internal
$m$-factors from $m_3$ through $m_r$.  The only difference is that
$C(\ell,m_3,\ldots,m_r;\mathbf n;k)$ contains the two boundary factors
\[
   \qbinom{\ell+m_3+1}{\ell+k}
   \quad\hbox{and}\quad
   \qbinom{m_r+\ell+1}{m_r+k},
\]
whereas $C(m_3,\ldots,m_r;\mathbf n;k)$ contains instead the closing factor
\[
   \qbinom{m_r+m_3+1}{m_r+k}.
\]
Expanding these three Gaussian coefficients in terms of $\qpoch{\cdot}$ gives
exactly \eqref{eq:C-recombine}.  Combining \eqref{eq:C-recombine} with the
prefactor in \eqref{eq:F-def} leaves the prefactor of the smaller expression
\[
   F_{r-1,s}^{(b-1)}(\ell,m_3,\ldots,m_r;\mathbf n;a;q).
\]
The factors that do not enter this smaller $F$ are
\[
   q^{\ell^2+\ell}\qbinom{m_1}{\ell}\qbinom{m_2+m_3+1}{m_2-\ell}.
\]
Thus \eqref{eq:deletion} follows.\qed

\noindent{\it Proof of Theorem~\ref{thm:Guo71}.}
It is convenient to prove the slightly stronger statement in which the $m_i$ are allowed to be non-negative, while the $n_j$ remain positive.  The stated theorem is the positive-parameter case.

We prove this stronger statement simultaneously for all admissible $r$ by
induction on $b$; equivalently, the induction assertion for a fixed $b$ is made
for every $r\geq b$.  For $b=1$, Lemma~\ref{lem:reciprocity} gives
\[
   F_{r,s}^{(1)}(\mathbf m;\mathbf n;a;q)
   =q^{\Delta(\mathbf m;\mathbf n)}
     F_{r,s}^{(r)}(\mathbf m;\mathbf n;s-a;q^{-1}).
\]
Since $0\leq s-a\leq s$, Guo's positivity theorem \eqref{eq:Guo-theorem-used} says that
$F_{r,s}^{(r)}(\mathbf m;\mathbf n;s-a;q)\in\NNq$.  Write
$F_{r,s}^{(r)}(\mathbf m;\mathbf n;s-a;q)=\sum_{i=0}^{d}c_iq^i$ with
$c_i\in\NN$ and $d\leq\Delta(\mathbf m;\mathbf n)$ by Lemma~\ref{lem:reciprocity}.  Then
\[
q^{\Delta(\mathbf m;\mathbf n)}
F_{r,s}^{(r)}(\mathbf m;\mathbf n;s-a;q^{-1})
=\sum_{i=0}^{d}c_iq^{\Delta(\mathbf m;\mathbf n)-i}\in\NNq .
\]
Hence the case $b=1$ is positive.

Assume now that $b\geq2$ and that the assertion is known for $b-1$.  If $r=2$, then necessarily $b=2$, and the desired assertion is exactly Guo's theorem \eqref{eq:Guo-theorem-used}.  If $r\geq3$, Lemma~\ref{lem:deletion} expresses $F_{r,s}^{(b)}$ as a finite sum whose coefficients
\[
   q^{\ell^2+\ell}\qbinom{m_1}{\ell}\qbinom{m_2+m_3+1}{m_2-\ell}
\]
belong to $\NNq$, multiplied by
\[
   F_{r-1,s}^{(b-1)}(\ell,m_3,\ldots,m_r;\mathbf n;a;q),
\]
which belongs to $\NNq$ by the induction hypothesis.  Therefore every term in \eqref{eq:deletion} is positive, and the induction step is proved.  This completes the proof.\qed

\section{Concluding remarks}\label{sec:concluding}

Theorem~\ref{thm:odd-Catalan} supplies the missing companion recurrence needed to prove the positivity of Guo's odd $q$-super Catalan numbers.  Guo had already obtained the recurrence for $C_{N+h,N}(q)$; Lemma~\ref{lem:backward-recurrence} gives the corresponding recurrence for $C_{N,N+h}(q)$, so that a strong induction on $m+n$ closes the proof.

Theorem~\ref{thm:Guo71} shows that Guo's alternating-sum positivity theorem is stable under replacing the coefficient $2r-1$ of $\bc{k}$ by any odd coefficient $2b-1$ with $1\leq b\leq r$.  The proof is structural: the reciprocity lemma supplies the endpoint $b=1$ from Guo's known endpoint $b=r$, and the deletion recurrence transfers positivity from $b-1$ to $b$.

It remains natural to ask whether analogous reciprocity and deletion mechanisms can be adapted to Guo's other conjectures on alternating sums involving $q$-Narayana-type factors.  The present arguments isolate two robust sources of positivity, but they do not by themselves settle those further conjectures.

\end{document}